\documentclass[preprint,9pt]{elsarticle}

\usepackage{amssymb}
\usepackage{amsmath}
\usepackage{amsthm}

\newtheorem{lemma}{Lemma}[section]
\newtheorem{thm}[lemma]{Theorem}

\newtheorem{cor}[lemma]{Corollary}

\newtheorem{hyp}[lemma]{Hypothesis}

\journal{Journal of Pure and Applied Algebra}
\begin{document}

\setlength{\parindent}{0mm}

\newcommand{\m}{$\,\textrm{max}\,$}
\newcommand{\w}{\widehat}
\newcommand{\wi}{\widehat}
\newcommand{\ov}{\overline}
\newcommand{\N}{\mathbb{N}}
\def \P{\mathbb{P}}

\newcommand{\E}{\mathcal{E}}
\newcommand{\K}{\mathcal{K}}
\newcommand{\sym}{\textrm{Sym}}
\newcommand{\A}{\textrm{Alt}}

\newcommand{\Z}{\mathbb{Z}}

\newcommand{\wt}{\widetilde}
\newcommand{\wh}{\widehat}
\newcommand{\ti}{\tilde}

\newcommand{\M}{$\textrm{M}$}
\newcommand{\J}{$\textrm{J}$}
\newcommand{\ch}{$\textrm{char}$}
\newcommand{\sy}{$\,\textrm{Syl}$}
\newcommand{\au}{$\textrm{Aut}$}
\newcommand{\PSL}{$\textrm{PSL}$}
\newcommand{\PSU}{$\textrm{PSU}$}
\newcommand{\PGL}{$\textrm{PGL}$}
\newcommand{\PGaL}{P\Gamma L}
\newcommand{\GL}{$\textrm{GL}$}
\newcommand{\GU}{$\textrm{GU}$}
\newcommand{\Sp}{$\textrm{Sp}$}
\newcommand{\PSp}{$\textrm{PSp}$}
\newcommand{\Sz}{$\textrm{Sz}$}
\newcommand{\SL}{$\textrm{SL}$}
\newcommand{\SU}{$\textrm{SU}$}
\newcommand{\F}{$\textrm{GF}$}
\newcommand{\C}{$\textrm{C}$}
\newcommand{\FO}{\textrm{fix}_{\Omega}}
\newcommand{\FL}{\textrm{fix}_{\Lambda}}
\newcommand{\FD}{\textrm{fix}_{\Delta}}
\newcommand{\fixO}{\textrm{fix}_{\Omega}}
\newcommand{\fixL}{\textrm{fix}_{\Lambda}}
\newcommand{\out}{$\textrm{Out}$}
\newcommand{\Sym}{\textrm{Sym}}
\newcommand{\Alt}{\textrm{Alt}}
\newcommand{\rank}{$\textrm{r}$}
\newcommand{\He}{$\textrm{He}$}
\newcommand{\aut}{\textrm{Aut}}

\def \<{\langle }
\def \>{\rangle }
\def \L{\mathcal{L}}

\begin{frontmatter}

\title{Corrigendum and addendum to ``Transitive permutation groups where nontrivial elements have at most two fixed points''}

\author{Paula Hähndel and Rebecca Waldecker}

\begin{abstract}
This article revisits earlier work by the second author together with Kay Magaard. We correct
several little results and we briefly discuss why, fortunately, the errors hardly affect our
main theorems and in particular do not affect the classification of simple groups that act with fixity 2. 
As an addition to the submitted article, this
version also contains \texttt{GAP} code in a little appendix at the end.
\end{abstract}

\end{frontmatter}


\section{Introduction}

\vspace{0.2cm}
The motivation to investigate permutation groups that act with low fixity
stems from applications to Riemann surfaces.
Following Ronse (see \cite{Ro1980}), we say that \textbf{a group $G$ has fixity $k \in \N$ on a set $\Omega$} if and only if
$k$ is the maximum number of fixed points of elements of $G^\#$ on $\Omega$.
More background on our motivation and on applications can be found in \cite{MW2}.
This article revisits the analysis of the Sylow 2-structure in connection with point stabilisers, which corrects a previous error and leads to improved general structure results. We also explain why, fortunately, the flaw did not lead to mistakes in the classification of finite simple groups that act with fixity 2.

In addition to two corrections, we prove a modified, more compact version of Theorem 1.4 from~\cite{MW2}. The main difference is that in our hypothesis here, we ask for fixity exactly 2 instead of fixity at most 2, which implies that $G$ has even order. In (1) we allow more general 2-power indices than in the corresponding case in Theorem 1.4 of \cite{MW2}, and the statement in (5) collects the former cases (5) and (6). In the proof of Theorem \ref{main2} we will comment on this with more details.

\begin{thm} \label{main2}
Suppose that $G$ is a finite, transitive permutation
group with permutation domain $\Omega$. Suppose further that
 $G$ acts with fixity~$2$ on $\Omega$. Then
$G$ has even order and 
one of the following holds:

\begin{enumerate}

\item[(1)]
$G$ has a subgroup of 2-power index that is a Frobenius group.

\item[(2)]
$|Z(G)|=2$ and $G/Z(G)$ is a Frobenius group.

\item[(3)]
The point stabilisers are metacyclic of odd order. If $H$ is a
nontrivial two point stabiliser, then $|N_G(H):H|=2$, $G$ is
solvable and $H$ or $N_G(H)$ has a normal complement $K$ in $G$ such
that $K$ is nilpotent and $(|K|,|H|)=1$.

\item[(4)]
The point stabilisers are metacyclic of odd order. Moreover $G$ has
normal subgroups $N,M$ such that $N < M < G$, $N$ is nilpotent,
$M/N$ is simple and isomorphic to $\PSL_2(q)$, to $\Sz(q)$ or to
$\PSL_3(4)$, and $G/M$ is metacyclic of odd order.


\item[(5)]
The point stabilisers have even order and $G$ has a normal subgroup
$N$ of odd order such that $O^{2'}(G)/N$ is either a dihedral or
semi-dihedral $2$-group or there exists a prime power $q$ such that
it is isomorphic to $\Sz(q)$ or to a subgroup of $\PGaL_2(q)$ that
contains $\PSL_2(q)$.

\end{enumerate}
\end{thm}

Throughout this article, we use \texttt{GAP}. We do not always refer to it in the bibliography, but we do it once, here, along with the extended version of this article that contains the code (\cite{GAP}, \cite{HW}).
We also want to note that some results here are proven in more generality than necessary for this corrigendum, because they are also relevant for ongoing work and we do not want to prove very similar results several times.

\begin{center}

\textbf{Acknowledgments}

\end{center}

We thank Stefan Hoffmann, Christoph Möller and Patrick Salfeld for pointing out mistakes, for suggesting possible improvements in \cite{MW2} and for fruitful discussions on this revision. 
We are grateful to the referee for suggesting additional details to improve this article.
We also thank Anika Streck for revising some results about the 2-structure, which contributed to several lemmas, and Chris Parker for making us aware of earlier work on low fixity actions. This article has been written remembering Kay Magaard, with much gratitude and appreciation.

\vspace{1cm}


\section{Preliminaries}

\vspace{0.2cm}

All groups in this article are meant to be finite, and we use standard notation for orbits and point stabilisers.
Let $\Omega$ be a finite set and suppose that a group $G$ acts on $\Omega$.
Then for all $\Delta
\subseteq \Omega$, all $g \in G$ and all $H \le G$ we let
$\FD(H):=\{\delta \in \Delta \mid \delta^h=\delta$ for all $h \in
H\}$ denote \textbf{the fixed point set of $H$ in $\Delta$} and abbreviate
$\FD(\<g\>)$ by $\FD(g)$.
For all $n \in \N$, we denote the cyclic group of order $n$ by
$\C_n$.

After the publication of \cite{MW2}, we became aware of Ronse's work, see \cite{Ronse}. In particular, as far as we know, the notion of fixity (as given in the introduction) goes back to him and generalises ideas by Bender, Buekenhout and Rowlinson, Hering, Hiramine and others, where the focus was on the maximum number of fixed points of involutions.

Throughout this work, we come across Frobenius groups, and therefore the following result is useful. A proof can be found in \cite{MW2}.

\begin{lemma}\label{charfrob}
Suppose that $G$ has a non-trivial proper subgroup $H$ such that the
following holds: ~Whenever $1 \neq X \le H$, then $N_G(X) \le H$.

Then $G$ is a Frobenius group with Frobenius complement $H$.
\end{lemma}

We now collect a few basic technical results from previous papers, and we prove them in a 
very general version here for future reference. 

\begin{lemma}\label{normaliser}

Suppose that $G$ is a finite, transitive permutation
group with permutation domain $\Omega$. Suppose further that
 $G$ acts with fixity~$k \in \N$ on $\Omega$ and let $\alpha \in \Omega$.

(i) If $1 \neq X \le G_{\alpha}$, then $|N_G(X):N_{G_\alpha}(X)| \le k$.
In particular, if
$\FO(X)=\{\alpha\}$, then
$N_G(X) \le G_\alpha$.

(ii) If $x \in G_{\alpha}^\#$, then $|C_G(x):C_{G_\alpha}(x)| \le k$. In particular, if
$\FO(x)=\{\alpha\}$, then
$C_G(x) \le G_\alpha$.

(iii) If $k=2$, then $G$ has even order.

(iv) $|Z(G)|$ divides $k$.

(v) If $p \in \pi(G)$ and $p>k$, then either $G_\alpha$ is a $p'$-group or it contains a Sylow $p$-subgroup of $G$.

\end{lemma}

\begin{proof}
Special cases of many of these statements have been proven in \cite{MW2}.
(i) and (ii) follow from the fact that $|\FO(X)|\le k$, because this implies that
$k \ge |\alpha^{N_G(X)}|=|N_G(X):N_{G_\alpha}(X)|$.

Now for (iii) we assume that $G$ has odd order and then we see that $G$ is a Frobenius group, with
(i) and Lemma \ref{charfrob}. Let $K$ denote the Frobenius kernel and let $H$ be a Frobenius complement.
By hypothesis we find $x \in G$ such that $|\FO(x)|=2$ and without loss $\alpha \in \FO(x)$.
If we let $X\le G_\alpha$ be the stabiliser of $\FO(x)$, then we see that $N_G(X)=G_\alpha=G_\beta$, because $G$ has odd order and then there is no element in $G$ that interchanges $\alpha$ and $\beta$.
This gives a partition of $\Omega$ into pairs, which is again impossible, because $G$ has odd order and acts transitively.

For (iv) we take an element $x \in G_\alpha$ that fixes exactly $k$ points, and we use (ii). Since $Z(G) \cap G_\alpha=1$, this implies 
that $|Z(G)|$ divides $k$.
Statement (v) follows from (i) and the fact that proper subgroups of $p$-groups are properly contained in their normaliser.
A special case of this has been proven as Lemma 2.11 (a) in \cite{MW2}.
\end{proof}

Part (v) of this lemma shows that if $k=2$, $p \in \pi(G)$ is odd and $P\in \sy_p(G)$, then all $P$-orbits
have length $1$ or $|P|$.
For the prime 2 the situation is more interesting, and this is where a mistake happened in Lemma 2.12 of \cite{MW2}. we will comment on this in the next section.
First,a preparatory, probably well-known lemma:

\begin{lemma}\label{anika1}
Suppose that $p$ is a prime number and that $P$ is a finite $p$-group of size $p^n$, where $n \in \N$ and $n \ge 3$. Suppose further that $x \in P$ is such that $|C_P(x)| \le p^2$.
Then $P$ has maximal class, i.e. nilpotency class $n-1$.
\end{lemma}

\begin{proof}
Of course $Z(P)\langle x \rangle \le C_P(x)$, and we have two cases:
$x \in Z(P)$, which gives that $P \le C_P(x)$ and leads to a contradiction, or 
$x \notin Z(P)$, whence $|Z(P)\langle x \rangle| \ge p^2$.
We conclude that $|C_P(x)| = p^2$ and hence $|x^P|=p^{n-2}$.
Then Satz III.14.23 from \cite{Hupp} gives the result.
\end{proof}

\vspace{1cm}


\section{Finite simple groups that act with fixity 2}

We begin with our main hypothesis and then analyse the 2-structure.

\begin{hyp}\label{hyp2fix}
Suppose that $G$ is a finite, transitive permutation
group with permutation domain $\Omega$. Suppose further that
 $G$ acts with fixity~$2$ on $\Omega$.
\end{hyp}

The main theorem in \cite{Ronse} gives essentially enough information about the 2-structure for the classification of all finite simple groups that act with fixity 2.  
Unfortunately, we were not aware of Ronse's article when we wrote \cite{MW2}, and there is a mistake in our own lemma about the 2-structure,
Lemma 2.12 of \cite{MW2}. 
The lemma looks almost as Lemma \ref{anika2} below, which is a corrected version. But in Part (2) we overlooked the possibility that
point stabilisers could have index 2.

\begin{lemma}\label{anika2}
Suppose that Hypothesis \ref{hyp2fix} holds, that $\alpha \in \Omega$ and that $S \in \sy_2(G)$.
Then one of the following is true:

(1) $S_\alpha = 1$.

(2) $S$ is dihedral or semi-dihedral and $|S_\alpha| = 2$ or $|S:S_\alpha|=2$.

(3) $|S| \ge 4$, there is a unique $S$-orbit of size 2 and all other $S$-orbits are regular. In this case $G$ has a normal subgroup of index either 2 or $|S|$ that is a Frobenius group.

(4) $S \le  G_\alpha$, in particular $|\Omega|$ is odd. 
\end{lemma}

\begin{proof}
Suppose that (1) and (4) do not hold.
In particular $1 \neq S_\alpha \neq S$. 
Then $G_\alpha$ has even order, and the size of $\Omega$ is even as well, which implies that $|G|$ is divisible by $4$.
We use Lemma 2.11 of \cite{MW2} and have two possibilities now, given that (4) does not hold:
$S$ is dihedral or semi-dihedral and $|S_\alpha|=2$ 
or $|S_\alpha|=\frac{|S|}{2}$ and there exists some $\beta \in \Omega$ such that $\alpha \neq \beta$, $S_\alpha=S_\beta$ and some element in $S$ interchanges $\alpha$ and $\beta$. The first case is covered by (2), and we mention that these cases can coincide if $|S|=4$.
In the second case, we take $\beta \in \Omega$ with the corresponding properties and we study the possible orbit lengths for $S$.
Since $|S:S_\alpha|=2$ and there exists some $s \in S$ such that $\alpha^s=\beta$, we see that all elements in $S \setminus S_\alpha$ interchange $\alpha$ and $\beta$. Therefore $\alpha^S$ is an orbit of size 2.
Moreover, the elements in $S_\alpha^\#$ fix exactly two points in $\Omega$, namely $\alpha$ and $\beta$, and 
this means that the remaining $S_\alpha$-orbits on $\Omega$ are regular. Thus, for each $\omega \in \Omega \setminus \{\alpha, \beta\}$, the orbit $\omega^S$ has length $|S_\alpha|(=\frac{|S|}{2})$ or $|S|$.

First we suppose that there is some $\omega \in \Omega \setminus \{\alpha, \beta\}$
such that $|\omega^S|=\frac{|S|}{2}$. Then $|S_\omega|=2$ and a combination of Lemma \ref{normaliser}(ii) and Lemma \ref{anika1} gives that $S$ has maximal class. Then, by 3.11 in \cite{Hupp}, it is a quaternion, dihedral or semi-dihedral group. But the quaternion group case is impossible because $S_\omega \neq Z(S)$ is generated by an involution. This gives (2) again. 
Now we suppose that $\alpha^S=\{\alpha,\beta\}$ is the only non-regular $S$-orbit.
Let $x \in S\setminus S_\alpha$, $X:=\langle x \rangle$ and $m \in \N$ be such that $|X|=2^m$.
Also, let $\omega \in \Omega \setminus \{\alpha, \beta\}$.
Then $|\omega^S|=|S|$ and hence $|\omega^X|=2^m$, which gives two cases:

Case 1: $2^m=|S|$. Then $S$ is cyclic and $G$ has a normal $2$-complement $N$ by Burnside's Theorem. 
If $N$ has regular orbits, then $|\Omega|=|G:G_\alpha|=\frac{|S|\cdot |N|}{|G_\alpha|}=|S:S_\alpha| \cdot |\alpha^N|=2 \cdot |N|$, which implies that $N$ has exactly two regular orbits, and $x$ interchanges them.
In particular $\Omega=\alpha^N \cup \beta^N$.
Now $N \cdot S_\alpha$ is a subgroup of $G$ of index 2, i.e. a normal subgroup, that acts as a Frobenius group on $\alpha^N$. This is included in (3).
If $N$ does not have regular orbits, then it acts transitively and with fixity at most 2 on its orbits, and the orbits have length at least 3. Fixity 2 is not possible because $N$ has odd order, and therefore $N$ acts as a Frobenius group on its orbits. Given that $|N|=|G:S|$, this leads to (3) again.

Case 2: $2^m<|S|$. Then $x$ has an even number of regular orbits on $\Omega \setminus \{\alpha, \beta\}$, and it interchanges $\alpha$ and $\beta$, and therefore it acts like an odd permutation on $\Omega$.
This means that the image of $G$ in $\Sym_{\Omega}$ is not contained in the alternating group, and therefore $G$ has a normal subgroup $M$ of index 2 that does not contain $x$. 
We see that $G=M \cdot X$ and we let $T:=M \cap S$. Then $|S:T|=2$, all elements in $T^\#$ act as even permutations on $\Omega \setminus \{\alpha, \beta\}$ and fix $\alpha$ and $\beta$, and in particular $T=S_\alpha$.
Together with the fact that $|G:M|=2$, this implies that $M$ does not act transitively on $\Omega$, but it has two orbits of equal size that are interchanged by $x$. In particular $G_\alpha \le M$,
and we recall that $S_\beta=S_\alpha \le G_\alpha \le M$ fixes two points. Together with the global fixity 2 
hypothesis, this means that $M$ acts as a Frobenius group on its orbits. Again, this is (3).
\end{proof}

As a direct consequence, we see:

\begin{cor}\label{orbitlength}
Suppose that Hypothesis \ref{hyp2fix} holds and that $S \in \sy_2(G)$.
Then the orbits of $S$ on $\Omega$ have length 
$1,~ 2, ~\frac{|S|}{2}$ or $|S|$.
\end{cor}

Based on our results on the $2$-structure we have a natural case distinction depending on whether the point stabilisers
have even or odd order.
Revisiting our results from \cite{MW2}, we first look at specific groups (and series of groups) that might appear as examples with point stabilisers of even order.
Then, for the remaining groups, we only need to consider the case where point stabilisers have
odd order, and our discussion will be brief because we can directly quote results from \cite{MW2} for this.


\begin{lemma}\label{A7}
Suppose that $G=\A_7$ or $G=\M_{11}$. Then
there is no set $\Omega$ such that $(G, \Omega)$ satisfies
Hypothesis \ref{hyp2fix}.

If $\Omega$ is a set such that $(\PSL_3(4),\Omega)$ satisfies Hypothesis \ref{hyp2fix}, then
the point stabilisers have order 5.
\end{lemma}

\begin{proof}
This can be checked using \texttt{GAP} (see \cite{GAP} and \cite{HW}).
\end{proof}

Next we analyse the general behaviour for some series of groups 
because they appear several times.

\begin{lemma}\label{PSL2}
Suppose that $p$ is an odd prime and that $m \in \N$. Let $q:=p^m \ge 5$, let
$G:=\PSL_2(q)$ and suppose that $(G,\Omega)$ satisfies
Hypothesis \ref{hyp2fix}. Then one of the following holds:

\begin{enumerate}
\item[(1)]
$q=5$, $|\Omega| \in \{5, 6, 10,12,20, 30\}$, and the point stabilisers in $G$ have structure $\A_4$, $D_{10}$,
$D_6$, $\C_5$, $\C_3$ or $\C_2$, respectively.

\item[(2)]
$q=7$, $|\Omega|=14$ and the point stabilisers are isomorphic to $\A_4$.

\item[(3)]
$|\Omega| =q+1$ and $G$ acts on $\Omega$ as on the set of cosets of
the normaliser of a Sylow $p$-subgroup of $G$.

\item[(4)]
$|\Omega| = q(q-1)$ and $G$ acts on $\Omega$ as on the set of cosets
of a cyclic subgroup of order $\frac{q+1}{2}$ (if $p$ is odd) or $q+1$ (if $p=2$).

\item[(5)]
$|\Omega|=q(q+1)$ and $G$ acts on $\Omega$ as on the set of cosets
of a cyclic subgroup of order $\frac{q-1}{2}$ (if $p$ is odd) or $q-1$ (if $p=2$).

\end{enumerate}
\end{lemma}

\begin{proof}
For the subgroup
structure of $G$ we refer to Theorem 6.5.1 in \cite{GLS3}, and we notice the
special cases $\PSL_2(4) \cong \PSL_2(5) \cong \A_5$ and $\PSL_2(9) \cong \A_6$.
First we look at the small groups, and we will see that they give some extra examples whereas the larger groups exhibit generic behaviour.
Checking $\PSL_2(5)$ with \texttt{GAP}, we obtain the possibilities that are listed, bearing in mind the subgroup structure of the group.
If we check $\PSL_2(7)$ with \texttt{GAP}, then the number 12 appears twice for the order of a point stabiliser in a fixity 2 action because there are two conjugacy classes
of subgroups isomorphic to $\A_4$ that appear as point stabilisers. This case is (2). The other possibilities with point stabilisers of order 3, 4, and 21, respectively, are special cases of the statements in (5), (4), and (3).

For the generic behaviour we quote Lemma 3.11 of \cite{MW2}, which only depends on properties in Lemma \ref{normaliser} and not on the flawed Lemma 2.12 in \cite{MW2}.
\end{proof}

\begin{lemma}\label{Suzuki}
Suppose that $m \in \N$ is odd, that $q:=2^m \ge 8$ and that $G=\Sz(q)$. 
Suppose further that $(G,\Omega)$ satisfies
Hypothesis \ref{hyp2fix}.
Then one of the following holds:

(1) $|\Omega|= q^2 + 1$ and $G$ acts $2$-transitively in its natural action, or

(2) $|\Omega|= q^2(q^2 + 1)$ and $G$ acts as on the set of cosets of a subgroup
of order $q - 1$.
\end{lemma}

\begin{proof}
We can mostly argue as in \cite{MW2}, using Theorem 6.5.4 in \cite{GLS3} for the subgroup structure.
Let $\alpha \in \Omega$ and $H:=G_\alpha$. If $H$ has even order, then we
take the $2$-structure of $G$ into account and the fact that $G$ is simple, and we see that Case (4) of Lemma \ref{anika2} holds.
Let $S \in \sy_2(G)$.
Then Lemma \ref{normaliser}(i) gives that a subgroup of index 2 of the maximal subgroup $N_G(Z(S))$ is contained in $H$, 
and also $|G:H|$ is odd, and therefore $N_G(Z(S))$ is a point stabiliser. This leads to (1). 

If $H$ has odd order, then it is a Hall subgroup of $G$, by Lemma \ref{normaliser}(v), and this leads to possibility (2).
\end{proof}

We do not discuss the series $\PSU_3(q)$ again, because we can quote Lemma 3.15 from \cite{MW2}. It
only depends on the subgroup structure and Lemma \ref{normaliser}. 
For the same reason we do not revise the analysis of the series $\PSL_3(q)$ in odd characteristics and with point stabilisers of even order (see Lemma 3.16 in \cite{MW2}). 
Small alternating and symmetric groups can be investigated with \texttt{GAP}, and the discussion in \cite{MW2} mostly does not depend on Lemma 2.12 there, which is why we can quote results directly. There is one exception, namely the central extension $2\A_n$, where $n \in \N$.
Therefore we revisit these groups here.

The following lemma is a preparation and we prove it in more generality than necessary, for future applications.

\begin{lemma}\label{centre}
Suppose that $G$ is a finite, transitive permutation
group with permutation domain $\Omega$. Suppose further that
 $G$ acts with fixity~$k \in \N$ on $\Omega$, that
$Z \le Z(G)$ and that $|\Omega| >k \cdot |Z|$.
Then $G/Z$ acts transitively,
non-regularly, and with fixity at most $k$ on the set of $Z$-orbits, with point stabilisers of the same size as in $G$.
\end{lemma}

\begin{proof}
We let $\bar \Omega$ denote the set of $Z$-orbits on $\Omega$ and we set $\bar G:=G/Z$.
Then $\bar G$ acts transitively on $\bar \Omega$.
We also recall that $G$ acts with fixity $k$ and hence $|Z|$ divides $k$, by Lemma \ref{normaliser}(iv), and $Z$ acts semi-regularly on $\Omega$.
Therefore, all elements of $\bar \Omega$ have size $|Z|$.
Let $\alpha \in \Omega$, $\bar \alpha:=\alpha^Z$ and $x \in G_\alpha^\#$.
Then $x$ stabilises $\bar \alpha$ and therefore $\bar x$ stabilises $\bar \alpha$, which means that
$\bar G$ does not act regularly on $\bar \Omega$.

Next we prove that $\bar G$ acts with fixity at most $k$, hence we let $g \in G$ and we suppose that
$\bar g$ fixes $k+1$ distinct points $\bar \omega_1$,..., $\bar \omega_{k+1}$ on $\bar \Omega$.
Let $i \in \{1,...,k+1\}$. Then the image $\omega_i^g$ is contained in $\bar \omega_i$, and we choose $z \in Z$ such that $\omega_i^g=\omega_i^z$. Now $gz^{-1} \in G_{\omega_i}$.
Since $|Z|$ divides $k$, there must be a set $J \subseteq \{1,...,k+1\}$ of size at least
$\frac{k}{|Z|} + 1$ and an element $z \in Z$ such that, for all $j \in J$, $\omega_j^g=\omega_j^z$.
In particular $gz^{-1}$ fixes $\omega_j$, and then it follows that it fixes all elements in
the orbit $\bar \omega_j$. However, this gives at least $|J|\cdot |Z|\ge
(\frac{k}{|Z|} + 1) \cdot |Z|$ fixed points in total, i.e. more than $k$.
This forces $g z^{-1} =1$,  hence $g \in Z$ and $\bar g=\bar 1$ acts trivially on $\bar \Omega$.

Finally, we look at the point stabilisers. Let $U \le G$ be such that $\bar U$ is a point stabiliser in $\bar G$.
Let $\omega \in \Omega$ be such that $\bar U$ stabilises $\bar \omega$. Then $G_\omega \le UZ$ and $G_\omega \cap Z=1$ because $Z \le Z(G)$.
Now
$$\frac{|\bar G|}{|\bar U|}=|\bar G:\bar U|=|\bar \Omega|=\frac{|\Omega|}{|Z|}=\frac{|G:G_\omega|}{|Z|}=|\bar G| \cdot \frac{1}{|G_\omega|}$$ 
and therefore
$|G_\omega|=|\bar U|$.
\end{proof}

We note that the previous lemma implies Lemma 2.14 of \cite{MW2} (in the special case where $|Z|=|Z(G)|=k=2$).

\begin{lemma}\label{alt}
Suppose that $n \in \N$, that $n \ge 5$ and that $G=2\A_n$. Then there is no set
$\Omega$ such that $(G,\Omega)$ satisfies Hypothesis \ref{hyp2fix}.
\end{lemma}

\begin{proof} Assume otherwise and let $\bar G:=G/Z(G)$. 
We recall that $G$ acts faithfully on $\Omega$ and with fixity $2$, 
and therefore
$|\Omega| \ge 4$. If 
$|\Omega| = 4$, then the faithful action gives that a section of $G$ is isomorphic to 
a subgroup of $\Sym_4$. This is impossible because $\bar G$ is simple and non-soluble.
Now we conclude that $|\Omega| \ge 5$
and that Lemma \ref{centre} is applicable.
In particular $\bar G$ satisfies Hypothesis \ref{hyp2fix}, with point stabilisers in $G$ and $\bar G$ of the same order. Then Lemma 3.7 and Theorem 3.8 in \cite{MW2} imply that
$\bar G \cong \A_5$ or $\A_6$. Lemma \ref{PSL2} gives the exact possibilities for the action, respectively, bearing in mind that 
$\A_5 \cong \,\PSL_2(5)$ and $\A_6 \cong \,\PSL_2(9)$.
But none of these are possible in $G$, with point stabilisers of the same order, because of Lemma \ref{normaliser}(i) and the fact that $|Z(G)|=2$.
\end{proof}

For many sporadic groups, it is possible to check whether or not they can act satisfying Hypothesis \ref{hyp2fix}
by using \texttt{GAP}. But we also revisited the arguments in the corresponding section of \cite{MW2} and give an alternative proof of 
Lemma 4.2 there, proving that potential examples would have point stabilisers of odd order. With this information, the remaining results (Lemma 4.3 up to Corollary 4.8 in \cite{MW2}) are applicable.


\begin{thm}\label{spor}

Suppose that $G$ is a sporadic finite simple group and that $\Omega$ is such that
$(G,\Omega)$ satisfies Hypothesis \ref{hyp2fix}. Then the point stabilisers have odd order.

\end{thm}

\begin{proof}
Assume otherwise. 
Then we apply Lemma \ref{anika2}, where, for a suitable choice of $\alpha \in \Omega$ and $S \in \sy_2(G)$, 
 one of the cases (2) or (4) must hold. 
For (2) we notice that sporadic simple groups with dihedral or semi-dihedral Sylow $2$-subgroups are known (see \cite{GW} and \cite{ABG}), and the only possibility, the Mathieu group $\M_{11}$, has already been excluded by Lemma \ref{A7}.
Thus we know that $G_\alpha$ contains a Sylow $2$-subgroup -- without loss $S \le G_\alpha$. 
Then we prove that $G$ has a strongly embedded subgroup, and we argue directly here, briefly, rather than quoting several results from \cite{MW2}. We notice that $|\Omega|$ is odd and we let $s \in S$ be an involution. Then $\alpha$ is the unique fixed point of $s$ on $\Omega$, and Lemma \ref{normaliser}(ii) forces $C_G(s) \le G_\alpha=:H$.
If $g \in G\setminus H$ and if $x\in H \cap H^g$ is a $2$-element, then $x$ fixes two distinct points on a set of odd size, which together with our global fixity 2 hypothesis forces $x=1$. Now $H$ is strongly embedded in $G$, and we apply Bender's main result in \cite{Ben}:
Since $G$ is non-abelian simple, Burnside's $p$-complement Theorem and the Brauer-Suzuki Theorem yield that
$S$ is neither cyclic nor quaternion. Then \cite{Ben} gives a list of possibilities for $G$, and
none of these groups is a sporadic simple group.
\end{proof}

Now we can give a new proof for the main result in \cite{MW2} about simple groups that act with fixity 2, and
thus confirm that it is correct in spite of the flawed Lemma 2.12. 
The reason why we can often still argue as in \cite{MW2} is that, in many cases, Lemma 2.12(2) was only used for the structure of Sylow $2$-subgroups, not for further details. Since the structure was correct (dihedral or semi-dihedral), these applications only needed careful double-checking and always turned out to be correct.
 
Details about the action of the groups satisfying Hypothesis\ref{hyp2fix} can be found in Lemmas 
\ref{A7}, \ref{PSL2} and \ref{Suzuki}.

\begin{thm}\label{fix2simple}
Suppose that $G$ is a finite simple non-abelian group and that $(G,\Omega)$ satisfies Hypothesis \ref{hyp2fix}.
Then $G$ is isomorphic to $\PSL_3(4)$ or there exists a prime
power $q$ such that $G$ is isomorphic to $\PSL_2(q)$ or to $\Sz(q)$.
\end{thm}

\begin{proof}
We argue along the cases of Lemma \ref{anika2}.
If the point stabilisers have odd order, then we can follow the arguments in \cite{MW2}, because this case is not affected by the flaw in Lemma 2.12 there. Hence we
use the CFSG and go through all series of groups (see Theorem 3.20 in \cite{MW2})
and all sporadic groups (see Theorem 4.7 in \cite{MW2}). As a result we find exactly those examples that are listed in the theorem. 

If $G$ has dihedral or semi-dihedral Sylow $2$-subgroups, as in Case (2), then we know
the possibilities for $G$ by the main results in \cite{GW} and \cite{ABG}. Combining this with Lemmas 3.15 and 3.16 from \cite{MW2} and  Lemmas
\ref{A7}, \ref{PSL2} and \ref{Suzuki} here this leads to the series $\PSL_2(q)$ and $\Sz(q)$.
Case (3) cannot occur because $G$ is simple, and in Case (4) we refer to the analysis in \cite{MW2}, which is similar to the arguments at the end of the proof of Theorem \ref{spor} above:
Since $|\Omega|$ is odd, Lemma 2.18 of \cite{MW2} applies and gives that $G$ is isomorphic to $\PSL_2(q)$, $\Sz(q)$ or $\PSU_3(q)$, where $q$ denotes a prime power. Finally,  Lemma 3.15 of \cite{MW2} shows that the third type of group does not occur under hypothesis \ref{hyp2fix}.
\end{proof}

\vspace{1cm}


\section{Technical results and the proof of Theorem 1.2}

In this section we revisit some results from \cite{MW2} that have not been necessary for the classification of finite simple examples, but that play a role for more general structure results.
Sometimes we revisit them and add more details, or we make small corrections or double-check the proofs because they depend on the $2$-structure, or we generalise results.

First we discuss  Lemma 2.16 from \cite{MW2}. Fortunately, the mistake we made there does not have any consequences in the remainder of the article \cite{MW2}, but we still feel that we should specify the problem and correct the mistake. Towards the end of the proof, when $(O_p(G)G_\alpha, \alpha^{O_p(G)})$ satisfies Hypotheses 1.1 of \cite{MW2} and $p=2$, then in the case where $Z(O_2(G)G_\alpha)=1$, it is not necessarily true that $G_\alpha$ acts fixed point freely on $O_2(G)$. An example for this situation is the group $G$ with id [96,70] in the GAP package SmallGrp~\cite{SmallGrp}. This group $G$ has trivial centre, $O_2(G)\cong C_2^4\rtimes C_2$ and $G_\alpha\cong C_3$.\label{sec:ex:fix2mistake}
It has a subgroup of index 2 of structure ($C_2^4\rtimes C_3$) that is a Frobenius group.

Here we have a corrected version:

\begin{lemma}\label{2.16}

Suppose that Hypothesis \ref{hyp2fix} holds and that $\alpha\in\Omega$. Suppose that $p\in\pi(G)$ is such that $O_p(G)\neq 1$ and that $G_\alpha$ has odd order. Then $G_\alpha$ is metacyclic.
More precisely:

If $p$ is odd, then $O_p(G)G_\alpha$ is a Frobenius group with complement $G_\alpha$. 

If $p=2$, then we set $P:=O_2(G)$ and
$Z:=Z(PG_\alpha)$, and then one of the following holds:
	\begin{itemize}
\item[(1)] $Z \neq 1$ and $G/Z$ is a Frobenius group.

\item[(2)] $Z=1$ and $PG_\alpha$ is a Frobenius group.

\item[(3)] There is some $\beta\in\Omega$ such that $1\neq G_\alpha\cap G_\beta$ is properly contained in $G_\alpha$, and $G_\alpha$ is a Frobenius group with complement $G_\alpha\cap G_\beta$.

\item[(4)] $G_\alpha$ fixes two points in $\Omega$. Then $P=Z(G)$ has order $2$ or $[P,G_\alpha]\cdot G_\alpha$ is a Frobenius group.
	\end{itemize}
\end{lemma}

\begin{proof}
First we observe that $O_p(G)G_\alpha$ acts transitively and
non-regularly on $\alpha^{O_p(G)}$. If $|\alpha^{O_p(G)}| \le 2$,
then $p=2$ and $O_2(G)=Z(G)$. Then Lemma 2.14 of \cite{MW2} implies that
$G/Z(G)$ is a Frobenius group. 
Otherwise $|\alpha^{O_p(G)}| > 2$, and then our fixity hypothesis implies that
$O_p(G)G_\alpha$ acts faithfully on this orbit.
Hence we may now suppose that
$(O_p(G)G_\alpha,\alpha^{O_p(G)})$ satisfies Hypothesis \ref{hyp2fix}.

If $p$ is odd, then Lemma \ref{normaliser}(v) yields that $p \notin \pi(G_\alpha)$, because $O_p(G) \neq 1$. Thus $C_{O_p(G)}(G_\alpha) = 1$ by Lemma \ref{normaliser}(ii), and it follows that $O_p(G)G_\alpha$ is a Frobenius
group with Frobenius complement $G_\alpha$. Since $G_\alpha$ has odd order, this means that $G_\alpha$ is metacyclic (see \cite{Hupp}, V 8.18(b)). 
Now suppose that
$p=2$, with notation as in the statement of our lemma. If $Z \neq 1$, then $PG_\alpha/Z$ is
a Frobenius group with complement isomorphic to $G_\alpha$, again by Lemma 2.14 of \cite{MW2}. 

\textbf{Case 1:} $Z = 1$ and $G_\alpha$ acts fixed point freely on $P$. 

Then $PG_\alpha$ is a Frobenius group with complement
$G_\alpha$. As in the previous paragraph we deduce that $G_\alpha$ is metacyclic.

\textbf{Case 2:} There exists some $g\in G_\alpha$ such that $C_{P}(g)\neq 1$. 

Then $|C_{P}(g)|=2$ by Lemma \ref{normaliser}(ii) and because $G_\alpha$ has odd order. Hence there exists some $\beta\in\Omega$ such that $G_\alpha\cap G_\beta\neq 1$. 

\textbf{Case 2a:} $G_\alpha\cap G_\beta$ is properly contained in $G_\alpha$.

Then Lemma 2.15 of \cite{MW2} gives that $G_\alpha$ is a Frobenius group with Frobenius complement $H:=G_\alpha\cap G_\beta$. It remains to prove that $G_\alpha$ is metacylic in this case. We notice that the kernel $K \le G_\alpha$ of this Frobenius group consists of $1$ together with all elements whose only fixed point is $\alpha$.
Then, since $P \cap G_\alpha=1$, we deduce for all $x \in K^\#$ that $C_P(x) = 1$.
This means that $PK$ is a Frobenius group with complement $K \le G_\alpha$ of odd order.
The fact that $K$ is a Frobenius kernel in the group $G_\alpha$ implies that $K$ is nilpotent, and consequently $K$ is nilpotent of odd order and with cyclic Sylow subgroups. Hence $K$ is cyclic.
Moreover $H$ and $K$ have coprime orders and $H$, being a Frobenius complement of odd order, also has cyclic Sylow subgroups. Then it follows that
$KH = G_\alpha$ is metacyclic.

\textbf{Case 2b:}  $G_\alpha=G_\beta$.

If $G_\alpha$ does not act faithfully on $P$, then Lemma \ref{normaliser}(ii) forces $|P|=2$, and then Lemma 2.14 of \cite{MW2} is applicable again.
First suppose that the action of $G_\alpha$ on $P$ is faithful. Then we set $H:=PG_\alpha$ and $\bar H:=H/\Phi(P)$.
Now $\bar G_\alpha$ also acts faithfully on $\bar P$ (because they have coprime order).
Coprime action and the fact that $\bar P$ is abelian yield that
$\bar P=[\bar P,\bar G_\alpha] \times C_{\bar P}(\bar G_\alpha)$.
But we also know that $P=[P,G_\alpha] \cdot C_{P}(G_\alpha)$ and that $|C_{P}(G_\alpha)| \le 2$ by Lemma \ref{normaliser}(ii). 
Therefore, if $[P,G_\alpha] \neq P$, then $[P,G_\alpha] \cdot G_\alpha$ is a Frobenius group with complement $G_\alpha$, as stated.
We assume for a contradiction that $[P,G_\alpha] = P$. Then $C_{P}(G_\alpha) \le [P,G_\alpha]$ and
the coprime action of $G_\alpha$ on $\Phi(P)$ implies that
$C_{\bar H}(\bar G_\alpha) = \overline{C_H(G_\alpha)}$, and in particular
$$C_{\bar P}(\bar G_\alpha)=\overline{C_{P}(G_\alpha)} \le \overline{([P,G_\alpha])}=[\bar P,\bar G_\alpha]$$ 
by our assumption.
This is false, hence $[P,G_\alpha] \neq P$ and our claim holds. 
In all cases V 8.18(b) in \cite{Hupp} yields that
$G_\alpha$ is metacyclic.
\end{proof}

In our discussion of components in \cite{MW2}, we felt in hindsight that we should have given more details in some of the proofs. 
All results are correct, but we prove a lemma here that is helpful in making some of the arguments much more clear later on. 
Again we prove it in more egenerality for future applications.

\begin{lemma}\label{Ex}
Suppose that Hypothesis \ref{hyp2fix} holds and that
$E$ is a component of $G$. If $x\in N_G(E) \setminus E$ is an involution,
then $|C_E(x)|\ge 3$.
\end{lemma}

\begin{proof}
Assume for a contradiction that $|C_E(x)| \le 2$.
Let $L:=E\langle x \rangle$ and let $x \in S \in \sy_2(L)$. As $E \unlhd L$, we have that $T:=S\cap E$ is a Sylow $2$-subgroup of $E$, that $T \unlhd S$ and that $Z:=Z(S) \cap T \neq 1$.
Moreover $Z \le C_E(x)$ and therefore $Z=C_E(x)$ and $|Z|=2$.
As $x \notin E$, this implies that $C_S(x)=Z\langle x\rangle$ is elementary abelian of order 4.
Then 5.3.10 in \cite{KS} yields that $S$ is a dihedral or semi-dihedral group, which makes $T$
a cyclic group, a dihedral group or a quaternion group. 
If $T$ is cyclic, then the simple group $E/Z(E)$ has cyclic Sylow $2$-subgroups, which is impossible.
Therefore $T$ is quaternion or dihedral and its image in
$E/Z(E)$ is a dihedral group.
Since $E/Z(E)$ is simple, the Gorenstein-Walter Theorem (main result in \cite{GW}) yields that $E/Z(E)$ is isomorphic to $\A_7$ or to $\PSL_2(q)$ for some odd prime power $q$.

Next we recall that $x$ normalises $E$, but it does not centralise $E$, and $E$ is perfect.
Now $x$ normalises $Z(E)$ and we briefly deduce why $x$ acts non-trivially on $E/Z(E)$: Otherwise $[E,x]\le Z(E)$ which, together with the Three Subgroups Lemma, implies that $[E,x]=1$. This is a contradiction.

It follows that the involution $x$ acts as an automorphism of order 2 on $E/Z(E)$.
Let $\bar E:=E/Z(E)$ and let $e \in E$ be such that $\bar e \in C_{\bar E}(x)$.
Then $[e,x] \in Z(E)$, which means that $e \in C_E(x)Z(E)$.
Our main assumption implies that $|C_E(x)Z(E)| \le 2 \cdot |Z(E)|$.

If $\bar E \cong \A_7$, then $Z(E)$ has order dividing 6 and
$x$ acts like a transposition or a triple transposition in $\Sym_7$.
Therefore $C_{\bar E}(x)$ has order 120 or 24, but $|C_E(x)Z(E)| \le 2 \cdot |Z(E)| \le 12$.
So this is impossible.

If $q$ is an odd prime power such that $\bar E \cong\, \PSL_2(q)$, then $Z(E)$ has order 2, generically,
only for $q=9$ it has order 6, because $\PSL_2(9) \cong \A_6$.
Therefore, if $q \ge 11$, then $|C_E(x)Z(E)| \le 2 \cdot |Z(E)| \le 4$, whereas
$|C_{\bar E}(x)| \ge \frac{q-1}{2}\ge 5$. This is impossible.
For smaller values of $q$ we have to analyse the situation more closely.

Suppose that $q=5$. Then $\bar E \cong \,\PSL_2(5) \cong \A_5$, $|Z(E)| \le 2$ and $x$ has six fixed points on $\bar E$. But $|C_E(x)Z(E)| \le 2 \cdot |Z(E)| \le 4$, so this is impossible.

Suppose that $q=7$. Then $\bar E \cong \,\PSL_2(7)$, $|Z(E)|\le 2$ and $x$ acts as a diagonal automorphism, so it has three fixed points on $\bar E$.
Now $|C_{\bar E}(x)|=3$, so if we let $C \le E$ denote the full pre-image in $E$, then it has order 3 (if $Z(E)=1$) or order 6 (if $Z(E) \neq 1$).
The first case is impossible by our assumption. In the second case we note that $Q:=O_3(C)$ has order 3 and is normalised by $x$, but not centralised (by our assumption on $C_E(x)$).
Therefore $x$ inverts it. Let $a\in Q$ be an element of order 3.
Then $a^x=a^{-1}=a^2$, and at the same time $[a,x] \in Z(E)$. As $[a,x]=a^{-1}a^x=a^2a^{-1}=a$, this forces $[a,x]=1$, which is again a contradiction.

Finally suppose that
$q=9$. Then $\bar E \cong \,\PSL_2(9)\cong \A_6$ and $|Z(E)|\le 6$. But $x$ has 10 or 24 fixed points on $\bar E$,
and both numbers are impossible for $|C_{\bar E}(x)|$.
\end{proof}

\begin{lemma}\label{E(G)}
Suppose that Hypothesis \ref{hyp2fix} holds, let $\alpha \in \Omega$ and suppose that $E(G) \neq 1$. Then $E(G)
\cap G_{\alpha}\neq 1$.
\end{lemma}

\begin{proof}
Assume that $E(G) \cap G_{\alpha}=1$ and let $x \in G_{\alpha}$ be
an element of prime order $p$. Let $E$ be a component of $G$
and assume that
$x$ does not normalize $E$.
The conjugates $E,E^x,...,E^{x^{p-1}}$ are distinct components of $G$ and therefore they commute. Consequently 
$L := \{ e \cdot e^x \cdots e^{x^{p-1}} \mid e \in E \}$ is a subgroup of $C_G(x)$
and $L$ is isomorphic to some section $E/Z$, where $Z \le Z(E)$.
Now Lemma \ref{normaliser}(ii) yields that $|L:L_\alpha|=|C_L(x):C_{L_\alpha}(x)| \le 2$, but
$L_\alpha \le E(G) \cap G_\alpha=1$, and therefore
$|L|=2$.
Since $L \cong E/Z$ and $E$ is a component, this is impossible.

It follows that $x$ normalises $E$ and still, by Lemma \ref{normaliser}(ii), we have that
$|C_E(x):C_{E_\alpha}(x)| \le 2$. As $E_\alpha=1$ by assumption, this implies that
$|C_E(x)| \le 2$.
If $x$ has odd order, then the main theorem in \cite{Fu2}
forces $E$ to be soluble, which is again a contradiction.
We are left with the case that $x$ is an involution, and then Lemma \ref{Ex} gives a contradiction.
\end{proof}

\begin{lemma}\label{onecomp}
Suppose that Hypothesis
\ref{hyp2fix} holds. If $\alpha\in \Omega$, then one of the following holds:

(1) All Sylow subgroups of $G_\alpha$ have rank 1 or

(2) $F^*(G)=E(G) O_2(G)$.

Moreover $G$ has at most one component, and if
$E(G) \neq 1$, then $O(G)=1$.
\end{lemma}

\begin{proof}
Suppose that (1) does not hold, that $p \in \pi(G_\alpha)$ and that
$P \in \sy_p(G_\alpha)$ has rank at least 2. Let $V \le P$ be elementary abelian of order $p^2$ and suppose that $r \in \pi(G)$ is such that $r$ is odd.
If we set $R:=O_r(G)$ and if $v \in V^\#$, then Lemma \ref{normaliser}(ii) yields that
$|C_R(v):C_{R_\alpha}(v)|\le 2$. Since $r>2$, this means that $C_R(v) \le G_\alpha$.
Coprime action yields that
$$R=\langle C_R(v) \mid v \in V^\#\rangle \le G_\alpha,$$
and then it follows that $R=1$.

This means that $F(G)=O_2(G)$, which is (2).
The next statement is Lemma 2.22 in \cite{MW2}.
Here we note that the proof of the corresponding results are based on results that we have partly revisited here, but they do not need to be changed and therefore we do not discuss them again here.

Finally, if $E(G) \neq 1$, then we let $E$ denote the unique component and we let $1 \neq x \in E_\alpha$, by Lemma \ref{E(G)}.
Then $O(G)$ has odd order and lies in $C_G(x)$, so by Lemma \ref{normaliser}(ii) it follows
that $O(G) \le G_\alpha$. This forces $O(G)=1$ because a non-trivial normal subgroup of $G$ cannot be contained in a point stabiliser.
\end{proof}

\vspace{0.5cm}

We conclude this article with the proof of Theorems \ref{main2}. 

\vspace{0.5cm}
\underline{\textbf{Proof of Theorem \ref{main2}:}}

We know by Lemma \ref{normaliser}(iii) that $G$ has even order.
Now we begin by following the arguments in \cite{MW2}, bearing in mind that we have fixity exactly 2 (rather than at most 2) and use our Lemma \ref{anika2} instead of Lemma 2.12 of \cite{MW2} where necessary. 
Assume that $G$ is a minimal counterexample to Theorem \ref{main2}, by which we mean that
$|G| + |\Omega|$ is minimal such that Hypothesis \ref{hyp2fix} holds, but that none of the cases from the theorem hold. Let $\alpha \in \Omega$.

If $G_\alpha$ is metacyclic of odd order and $H \le G_\alpha$ is a non-trivial two point stabiliser, then we argue as in Step 1 in the proof given in \cite{MW2} and we see that $H$ does not have a normal complement in $G$ (otherwise Case (3) of the theorem holds). Step 2 gives that $N_G(H)$ also does not have a normal complement, again because it would lead to Case (3) of the theorem. Since $G$ is not a Frobenius group, Theorem 5.6 in \cite{MW2} then forces Case (4) of Theorem \ref{main2} to hold. This is a contradiction.
Hence $G_\alpha$ is not metacyclic of odd order. In particular, it is not possible that $G_\alpha$ is contained in a Frobenius complement of odd order in some subgroup of $G$. 

Next we turn to Lemma
\ref{anika2}, replacing Lemma 2.12 of \cite{MW2}.
If $G_\alpha$ has odd order, then Lemma \ref{2.16} above is applicable.
If $F(G) \neq 1$, then $G_\alpha$ is metacyclic (contrary to the previous paragraph) or Case (2) of the theorem holds, which is impossible because $G$ is a counterexample.
Now $F(G)=1$ and Lemma \ref{onecomp} yields that $F^*(G)$ is quasi-simple. Together with Lemma \ref{E(G)}
we see that the unique component $E$ of $G$ does not act semi-regularly, and then the minimal choice of $G$ applies. 
Of course $E$ is not a Frobenius group, and hence it acts with fixity 2, and
Theorem 5.1 in \cite{MW2} tells us that $E$ must be simple. Now $G$ is almost simple, but 
not simple, because otherwise Case (4) of the theorem is true with our results from Section 3. 

If $G_\alpha$ has even order and $G$ has dihedral or semi-dihedral Sylow $2$-subgroups, then Case (5)
of our theorem holds, which is again a contradiction.

If $G_\alpha$ has even order and Case (3) of Lemma \ref{anika2} is true, then we immediately have Case (1)
of Theorem \ref{main2}, which is again impossible.

Finally, if $G_\alpha$ contains a Sylow 2-subgroup of $G$, then we refer to Lemma 2.18 from \cite{MW2}.
Since $G$ is not a Frobenius group, this Lemma leads directly to Case (5) of Theorem \ref{main2}.

We are left with the case that $G_\alpha$ has odd order and $F^*(G)$ is simple. 
Then $F^*(G) \neq G$ because $G$ is a counterexample, and Theorem 1.3 in \cite{MW2} gives another contradiction.


\vspace{1cm}

\section*{Appendix}

The following \texttt{GAP} code uses the table of marks and determines, for a positive integer \(k\) and the table of marks \verb|t| of a group \(G\), all faithful and transitive fixity-\(k\) actions of \(G\). The command \verb|TestTom(t,k)| returns a list that contains an entry for each fixity-\(k\) action of \(G\), represented by a list that has as first entry a description of the point stabiliser structure and as second entry a list of fixed point numbers. If there are no fixity-\(k\) actions, then an empty list is returned.
For many finite simple groups, the table of marks is already pre-computed and accessible through the package TomLib~\cite{TomLib}.

\begin{verbatim}
TestTom:=function(t,k)
  local marks,g,fin;
  marks:=MarksTom(t);;
  fin:=[];;
  for g in [1..Length(marks)] do
    if ForAll([2..Length(marks[g])],i->marks[g][i]<k+1)
        and (k in marks[g])
        and marks[g][1]>k
      then Add(fin,[StructureDescription(RepresentativeTom(t,g)),
        marks[g]]);
    fi;
  od;
  return fin;
end;
\end{verbatim}

Throughout the paper, this algorithm has been used for a number groups. For instance, it has been used to establish that none of the groups \(\Alt_7\) and \(\M_{11}\)
can act with fixity~\(2\). This can been seen by the fact that\\\
\verb|List(["A7","M11"],x->TestTom(TableOfMarks(x),2));|\\
returns a list of empty lists. Before these calculations, the package TomLib and the functionality \verb|TestTom| have to be loaded.

If a fixity-\(k\) action exists, then a \texttt{GAP} session can look as follows:
\begin{verbatim}
gap> TestTom(TableOfMarks("L2(5)"),2);
[ [ "C2", [ 30, 2 ] ], [ "C3", [ 20, 2 ] ], [ "C5", [ 12, 2 ] ],
[ "S3", [ 10, 2, 1, 1 ] ], [ "D10", [ 6, 2, 1, 1 ] ],
[ "A4", [ 5, 1, 2, 1, 1 ] ] ]
gap> TestTom(TableOfMarks("L2(7)"),2);
[ [ "C3", [ 56, 2 ] ], [ "C4", [ 42, 2, 2 ] ],
[ "A4", [ 14, 2, 2, 2, 2 ] ], [ "A4", [ 14, 2, 2, 2, 2 ] ],
[ "C7 : C3", [ 8, 2, 1, 1 ] ] ]
\end{verbatim}
This completes the proof of Lemma~\ref{PSL2}.

\bigskip
In Section~\ref{sec:ex:fix2mistake} we study a group with id [96,70] in the Small Groups Library~\cite{SmallGrp}. Some of its properties can be verified as follows:
\begin{verbatim}
gap> g:=SmallGroup([96,70]);;
gap> Center(g);
Group([  ])
gap>  List(LowIndexSubgroups(FittingSubgroup(g),2), IsElementaryAbelian);
[ false, false, false, false, true, false, false, false ]
gap> TestTom(TableOfMarks(g),2);
[ [ "C3", [ 32, 2 ] ] ]
\end{verbatim}



\begin{thebibliography}{0.2cm}


\bibitem{ABG}  Alperin, J. L., Brauer, R. and Gorenstein, D.:
 Finite groups with quasi-dihedral and wreathed Sylow 2-subgroups.
\textit{Trans. Amer. Math. Soc.} \textbf{151} (1970) 1--261.


\bibitem{Ben} Bender, H.:
Transitive Gruppen, in denen jede Involution genau einen Punkt festl\"a{\ss}t,
\textit{J. Algebra} \textbf{17} (1971), 527--554.



\bibitem{SmallGrp} Besche,  H.  U.,  Eick,  B., O'Brien, E. and Horn, M. (Apr. 2022).
\emph{SmallGrp, The GAP Small Groups Library, Version 1.5}. \texttt{https://gap-packages.github.io/smallgrp/}. GAP package.



\bibitem{Fu2} Fukushima, H.: Finite groups admitting an automorphism of prime order fixing an abelian $2$-group. \textit{J. Algebra} {\bf 89} (1984), no. 1, 1--23.




\bibitem{GW} Gorenstein, D. and Walter, J.H.: The Characterization of Finite Groups with
Dihedral Sylow 2-Subgroups. \textit{J. Algebra} \textbf{2} (1965),
85--151.



\bibitem{GLS3} Gorenstein, D., Lyons, R. and Solomon, R.: \textit{The
Classification of the Finite Simple Groups, Number 3}. Mathematical
Surveys and Monographs 40.3 (American Mathematical Society,
Providence, RI), 1998.


\bibitem{TransGrp} Hulpke,    A. (Jul. 2022)  \emph{TransGrp},   Transitive   Groups   Library,   Version   3.6.3. \texttt{https://www.math.colostate.edu/~hulpke/transgrp}. GAP   package.

\bibitem{Hupp} Huppert, B.: \emph{Endliche Gruppen I.} Die Grundlehren der mathematischen Wissenschaften in Einzeldarstellungen, Band 134. Springer, 1967.


\bibitem{KS} Kurzweil, H. and Stellmacher B.: \textit{The Theory of Finite Groups}. Springer, 2004.


\bibitem{MW2}
	Magaard, K. and Waldecker, R.: Transitive permutation
	groups where nontrivial elements have at most two fixed points, 
	\emph{J. Pure Appl. Algebra} 219.4, 729--759. 
	

\bibitem{TomLib}
	Merkwitz, T., L. Naughton, and G. Pfeiffer (Oct. 2019).
	\emph{{TomLib}, The GAP Library of Tables of Marks, {V}ersion
		1.2.9}. \texttt{https://gap-packages.github.io/}\texttt{tomlib}. GAP
	package.



\bibitem{Ro1980}
Ronse, Chr.: On Permutation Groups of Prime Power Order.
\textit{Math. Z.} \textbf{173}, 211--215 (1980).


\bibitem{Ronse}
	Ronse, Chr.: On finite permutation groups in which
	involutions fix at most {$15$} points. \textit{Arch. Math. (Basel)}
	39.2, 109–112. 

\bibitem{GAP} The GAP Group, GAP -- Groups, Algorithms, and Programming, Version 4.13.1; 2024.







\end{thebibliography}
\end{document}